\documentclass[a4manuscript,12pt]{article}

\usepackage{bm}
\usepackage{bbm}
\usepackage[utf8]{inputenc}
\usepackage[english]{babel}
\usepackage{amssymb}
\usepackage{amsmath}
\usepackage{amsthm}
\usepackage{amsfonts}
\usepackage{mathrsfs}
\usepackage{dsfont}
\usepackage{hyperref}
\usepackage{commath}
\usepackage{graphicx}
\usepackage{text comp}
\usepackage{mathtools}
\usepackage{afterpage}
\usepackage{appendix}

\usepackage{tabularx}
\usepackage{booktabs}
\usepackage[labelfont=bf,format=plain,justification=raggedright,singlelinecheck=false]{caption}

\theoremstyle{plain}
\newtheorem*{thm*}{Theorem}
\newtheorem{thm}{Theorem}[section]

\newtheorem{lemma}[thm]{Lemma}
\newtheorem*{lemma*}{Lemma}
\newtheorem{corollary}[thm]{Corollary}
\newtheorem*{corollary*}{Corollary}
\newtheorem{prop}[thm]{Proposition}
\newtheorem*{prop*}{Proposition}
\newtheorem{defn}{Definition}

\newtheorem*{conjecture*}{Conjecture}

\newcommand{\R}{\mathbb{R}}

\newcommand{\N}{\mathbb{N}}

\newcommand{\1}{\mathbbm{1}}

\everymath{\displaystyle}
\title{Which Sampling Densities are Suitable for Spectral Clustering on Unbounded Domains?}

\author{Henry-Louis de Kergorlay
\thanks{%
           School of Mathematics,
           University of Edinburgh,
           Edinburgh, EH9 3FD, UK
           (\texttt{hdekerg@ed.ac.uk}).
           }
}

\begin{document}

\maketitle
\begin{abstract}
We consider a random geometric graph with vertices sampled from a probability measure supported on $\R^d$, and study its connectivity. We show the graph is typically disconnected,
unless the sampling density has superexponential decay. In the later setting, we identify an asymptotic threshold value for the radius parameter of the graph such that, for radius values beyond the threshold, some concentration properties hold for the sampled points of the graph, while the graph is disconnected for radius values below the same threshold. Properties of point processes are well-known to be closely related to the analysis of geometric learning problems, such as spectral clustering. This work can be seen as a first step towards understanding the consistency of spectral clustering when the probability measure has unbounded support. In particular, we narrow down the setting under which spectral clustering algorithms on $\R^d$  may be expected to achieve consistency, to a sufficiently fast decay of the sampling density (superexponential) and a sufficiently slowly decaying radius parameter value as a function of $n$, the number of sampled points. 
\end{abstract}

\section{Introduction}\label{sec: intro}
The goal of this work is twofold. One objective is to study the connectivity of random geometric graphs on $\R^d$, and understand how such connectivity properties are affected by changing the decay of the tail of the sampling density. The second objective, which motivates the first one, is to uncover an appropriate setting under which one may be able to extend consistency results for geometric learning problems such as spectral clustering (e.g., \cite{consistencySpectralClustering,variationalApproach,optimalCheeger}), from bounded to unbounded domains. 

Random geometric graphs and their properties have been widely studied, with works originating in \cite{randomPlaneNetworks,sphereRandomCovering}. Connectivity properties of random geometric graphs are well understood in the setting where the domain is bounded (\cite{randomGraphsConnectivity,randomGraphsConnectivityStrongLaw}). See \cite{randGraphs} for a classic exposition on random geometric graphs, and \cite{randomGraphsSurvey} for a survey on the topic. In the setting where the domain is $\R^d$, both \cite{randGraphs} and \cite{gaussianNetworkConnectivity} investigate the special case where the sampling density is a Gaussian, and identify sharp connectivity thresholds. However, to our knowledge, random geometric graphs with other sampling densities on $\R^d$ have not been studied, and it is not known how the connectivity of the graphs are affected by the choice of sampling density on $\R^d$.


Clustering (i.e., grouping data points according to some affinities) is a central task in machine learning. Among the many procedures available for this task, spectral clustering is widely used and has proven highly efficient in unsupervised learning. See \cite{spectralClusteringTutorial} for an introductory survey. Given a learning algorithm, a natural and important problem is that of determining its consistency as the number of sampled points $n\to \infty$, so as to provide a guarantee of accuracy of the results for sufficiently large $n$. In the case of clustering algorithms, consistency asks that the obtained labelling of the points converges in some way to an underlying canonical partition of the domain, which one would be able to infer if given full information of the probability measure (rather than a sample). Consistency of spectral clustering may be shown by establishing spectral convergence of a graph Laplacian to an underlying continuous Laplacian on the domain (\cite{consistencySpectralClustering}). While the consistency theory of spectral clustering is well-understood when the domain is bounded (e.g., \cite{variationalApproach,errorEstimates}), little theory is known when the sampling density has unbounded support, despite this being a common scenario in practice.

\section*{Acknowledgements}
I would like to thank Desmond Higham for helpful comments on a previous version of this manuscript. This work was supported by the Engineering and Physical Sciences Research Council grant EP/P020720/1.
\section{Properties of Point Processes and Geometric Learning Problems}\label{sec: properties of geometric point processes}
Properties of random point processes are well-known to contain valuable information with regards to understanding the consistency theory of machine learning problems. For instance, when the domain is bounded, connectivity thresholds for random geometric graphs (\cite{randomGraphsConnectivity}) provide lower thresholds for the range of asymptotically admissible values for the radius parameter in spectral clustering algorithms (e.g., \cite{spectralClusteringTutorial,variationalApproach}). Indeed, if the graph is disconnected, then $\lim_n dim(ker(\Delta^{(n)}))>1=dim (ker(\Delta))$, where $\Delta^{(n)}$ and $\Delta$ denote respectively an underlying graph Laplacian and continous Laplacian, hence the graph Laplacian fails to converge spectrally to the continuous Laplacian.

Another example of relevance here is \cite{topoCrackle}, where the authors study limits of random point processes on $\R^d$ under general geometric constraints. A particular application they derive from their results, is the analysis of Betti numbers of random \v Cech complexes on $\R^d$, deciding whether one can learn the topology of a manifold from points sampled with unbounded noise (see also \cite{crackle}). If the tail of the sampling density does not decay sufficiently fast (not faster than exponential), then the authors show the emergence of noisy cycles, assuming no or very mild asymptotic conditions are satisfied; thus showing that one cannot recover the topology of a manifold in these cases. On the other hand, if the decay of the tail is faster than exponential, they provide asymptotic conditions to guarantee that the union of the random balls centered around the sampled points is contractible. This implies in particular the connectivity of the induced random geometric graph in this setting, and we may wonder how sharp the asymptotic conditions provided in \cite{topoCrackle} are, with regards to connectivity. Is there a sharp connectivity threshold, such that the graph is disconnected with high probability (w.h.p.) for radius values (the radius $r_n$ such that two points are connected by an edge if their distance is less than $r_n$) below the threshold, while the graph is connected w.h.p.\ for radius values beyond the same threshold? What happens to the connectivity of the graph if the decay of the tail is exponential or slower? 

These natural questions motivate us to proceed under a similar setting to \cite{topoCrackle}, considering various decays for the tail of the sampling density by means of regular varying functions. Our main objective is to understand the connectivity of random geometric graphs on $\R^d$, in order to gain new insights for a theory of consistency of spectral clustering on $\R^d$. We are thus interested in identifying (and discarding) regimes where graphs are disconnected with high probability, in which setting there is little hope for spectral clustering to achieve consistency (see \cite{spectralClusteringTutorial}, where it is generally advised against performing spectral clustering on disconnected graphs). We are also interested in identifying other regimes where concentration properties hold for the sampled points. These properties should resemble the concentration properties which were successfully used in previous works studying the consistency of spectral clustering (or other gemoetric learning problems), where the domain was bounded. Such concentration properties for the sampled points appear as a necessary component in order to show consistency of many geometric learning problems, as they allow one to approximate a continuous operator by a discrete one, something which is often required when investigating consistency. See for instance the use of the $\infty$-Wassertein distance in \cite{variationalApproach,errorEstimates} or Lemma $3.2$ in \cite{optimalCheeger}.

In this work, we find a dichotomy between the setting where the decay of the sampling density is superexponential (e.g., a Gaussian), and settings where the decay is slower. This dichotomy is in agreement with the results found in \cite{topoCrackle}. We show that random geometric graphs are disconnected with high probability under no or mild constraints on the radius parameter of the graph, if the decay of the tail of the sampling density is exponential or slower. On the other hand if the decay is superexponential, we identify an asymptotic threshold below which the graph is disconnected with high probability (answering a question raised in \cite{topoCrackle} about the sharpness of their asymptotic conditions for the contractibility of the random balls), while concentration properties for the sampled points are shown to hold for radius values beyond the same threshold.

\begin{table}
\caption{Notation table}
\begin{tabularx}{\textwidth}{@{}XX@{}}
\toprule
   $d\geq 2$ & Ambient dimension\\
   $q$ &  Sampling density supported on $\R^d$\\
   $\mathcal P_n$ & Poisson point process sampled with respect to $q$\\
   $r_n$ & Radius parameter of the graph\\
   $G(\mathcal P_n,r_n)$ & Random geometric graph\\

  \\
  \textbf{Symbols} \\
  \\
  $A_n$ is true with high probability (w.h.p.) & $\lim_{n\to \infty}\mathbb P(A_n)=1$\\
  \\
  $f(n)\sim g(n)$ & $\lim_{n\to \infty}\frac{f(n)}{g(n)}=1$\\
   \\
  $f(n)\lesssim g(n)$ & $\exists\ C>0,\ \exists\ n_0\in \N,\ \forall\ n\geq n_0,$\\ & $f(n)\leq C g(n)$ \\
   \\
  $f(n)=O(g(n))$
  & $\exists\ C>0,\ \exists\ n_0\in \N,\ \forall\ n\geq n_0,$\\ &$|f(n)|\leq C g(n)$ \\
   \\
  $f(n)=\Omega(g(n))$ & $\exists\ C>0,\ \exists\ n_0\in \N,\ \forall\ n\geq n_0,$\\ 
  &$f(n)\geq C g(n)$ \\
  \\
  $f(n) = \Theta (g(n))$ & $\exists\ C_2,C_1>0,\ \exists\ n_0\in \N,\ \forall\ n\geq n_0,$\\ 
  &$C_1 g(n)\leq f(n)\leq C_2 g(n)$ \\
  \\
  $f(n)=o(g(n))$ & $\forall\ C>0,\ \exists\ n_0\in \N,\ \forall\ n\geq n_0,$\\
  &$|f(n)|\leq C g(n)$\\
  \\
  $f(n)=\omega(g(n))$ & $\forall\ C>0,\ \exists\ n_0\in \N,\ \forall\ n\geq n_0,$\\
  &$|f(n)|\geq C |g(n)|$\\
\bottomrule
\end{tabularx}
\end{table}

\section{Setting}
\subsection{Random geometric graphs}

Let $d\geq 2$ and let $\nu$ be a probability measure supported on $\R^d$. We consider a homogenous Poisson point process $\mathcal P_n$ of intensity $n\nu$, sampled with respect to $\nu$. That is, 
$$
\mathcal P_n=\{x_1,\dots,x_N\},
$$
where $N\sim Po(n)$ is independent from the points $x_i$'s, and the $x_i$'s are i.i.d. samples with respect to $\nu$.

Given a random point cloud $\mathcal P_n$, and a radius $r_n>0$, also known as the bandwidth parameter, we can form the random geometric graph $G(\mathcal P_n,r_n)$, where two points of $\mathcal P_n$ are connected by an edge if their distance is less than $r_n$. We will assume that $r_n\leq 1$ and that $r_n$ is non-increasing as a function of $n$.

In practice, it is common to assign weights to the edges. This is done via a kernel function $\eta$ satisfying a few properties.

\begin{defn}\label{defn: radial}
\begin{itemize}
\item
A function $f:\R^d\to \R$ is called \textbf{radial} (or isotropic, or symmetric), if 
$$
\forall (x,y)\in \big(\R^d\big)^2,\ |x|=|y|\Rightarrow f(x)=f(y).
$$
\item Given a radial function $f:\R^d\to \R$, define its radial profile to be the function $\textbf{f}:\R_+\to \R$ satisfying
$$
\forall x\in \R^d,\ \textbf{f}(|x|)=f(x).
$$
\end{itemize}
\end{defn}

We may then define the weighted edges of the graph $G(\mathcal P_n,r_n)$ as\\ $\{\eta_r(x-y)\ |\ x,y\in \mathcal P_n,\ x\neq y\}$, where $\eta:\R^d\to \R_+$ is radial and $\eta_r(x):=r^{-d}\eta(x/r).$  Let us assume furthermore that $\bm{\eta}(0)>0$, and that $\bm{\eta}$ is non-increasing, continuous almost everywhere and compactly supported.

The most basic kernel choice is $\eta(z):=\1(|z|<c)$, for some constant $c>0$. We then recover unweighted random geometric graphs, which are exactly the graphs studied in \cite{randGraphs}. Another typical choice is for $\eta$ to have an inverse exponential shape. Provided the above conditions are satisfied, the particular choice of $\eta$ is not important and does not impact the results of this work. We shall thus assume from now on, without loss of generality, that the graph $G(\mathcal P_n,r_n)$ is unweighted.

\subsection{Regularly varying functions}
A classic approach to modelling probability measures of unbounded support is to use regularly varying functions. They provide a flexible framework, allowing us to span through the main different tail behaviors on $\R^d$ (see Section~\ref{sec: probability measures}),
and possess convenient properties which facilitate rigorous analysis. 
\begin{defn}\label{regularly varying on R}
Regularly varying functions on $\R_+$ are defined as follows.
\begin{itemize}
    \item We say that a function $f:\R_+^*\to \R_+^*$ is \textbf{regularly varying} (at $\infty$) with exponent $\alpha\in \R$, and we write $f\in RV_{\alpha}(\R_+^*;\R_+^*)$ or simply $f\in RV_{\alpha}$ when clear from context, if for all $t >0$
    $$
    \lim_{R\to \infty}\frac{f(tR)}{f(R)}=t^\alpha.
    $$
    \item If the exponent in the definition above is $\alpha=0$, we say that the function is \textbf{slowly varying}.
\end{itemize}
\end{defn}
  
  Following a standard convention, we shall denote slowly varying functions by $L$.
 Regularly varying functions satisfy several useful key properties. 
 
\begin{thm}[Characterization theorem for regularly varying functions (\cite{regularlyVaryingSequences})]\label{characterization thm}
For every regularly varying function $f$ on $\R_+$, there exist a slowly varying function $L$ and $\alpha\in \R$ such that
$$
\forall\ t>0,\ f(t)=L(t)t^\alpha.
$$
\end{thm}

\begin{thm}[\cite{regularVariation}]
Suppose that $f\in RV_\alpha$, $\alpha\neq -1$, that $f$ is differentiable everywhere and $f'$ is eventually monotone, then $f'\in RV_{\alpha-1}.$
\end{thm}

The notion of regular variation may be extended to $\R^d$, however we focus on radial sampling densities which, as functions supported on $\R^d$, are regularly varying if and only if their radial profile, supported on $\R_+$, is regularly varying. Hence the above definitions suffice for the set up considered in this work.


\subsection{Probability measures on $\R^d$}\label{sec: probability measures}
Let us assume that the probability measure $\nu$ from which we sample the point process $\mathcal P_n$ has a sampling density, denoted by $q$, which we shall always assume to be radial (see Definition~\ref{defn: radial}) and supported on $\R^d$. By a slight abuse of notation, we shall use $q$ to denote either the sampling density itself or its radial profile $\bf q$. We feel that this ambiguity is clear from context and facilitates the fluidity of the reading.

The first case we consider is when the point process $\mathcal P_n$ is sampled with respect to a heavy tail density.
\begin{defn}[Heavy tail]\label{heavy tail}
Let $q:\R^d\to \R_+^*$ be a radial sampling density. We say that $q$ has a \textbf{heavy tail} if there exists $\alpha>d$ such that $q\in RV_{-\alpha}$.
\end{defn}
A typical example of a heavy tail density is the power law
$$
q(x):=\frac{C}{1+|x|^\alpha},\ x\in \R^d,\ \alpha>d.
$$

We also consider cases where the sampling density has a lighter tail than heavy tails. Following standard set ups in extreme value theory (e.g., \cite{highRiskExtreme,topoCrackle}), we define light tail densities by means of regular varying and von Mises functions.

\begin{defn}
Given a function $\psi\in C^2(\R_+;\R_+)$, we say that it is of \textbf{von Mises type} if for all $z\to \infty,$
$$
\begin{cases}

 \psi'(z)>0,\\
 \psi(z)\to \infty,\\
 a'(z)\to 0, \text{ where } a(z):=\frac{1}{\psi'(z)}.

\end{cases}
$$

\end{defn}
If $\psi$ is of von Mises type, then $\psi'>0$ holds true asymptotically, hence $\psi$ has a well-defined inverse, asymptotically.
\begin{defn}
Given a function $\psi\in C^2(\R_+;\R_+)$ of von Mises type, denote by $\psi^\leftarrow$ any function that is asymptotically the inverse of $\psi$.
\end{defn}

\begin{defn}[Light tail]\label{defn: light tail}
We say that a radial sampling density $q:\R^d\to \R_+^*$ has a \textbf{light tail} if it can be written in the form 
$$
q(x):=L(|x|)\exp\big(-\psi(|x|)\big),\ x\in \R^d,
$$
where $\psi\in C^2(\R_+;\R_+)$ is of von Mises type,  $\psi'$ is eventually monotone,  and $L\in RV_\alpha$ for some $\alpha\in \R$.



\end{defn}

Note that we may rewrite a light tail density as
$$
q(x)=\exp\left(-(\psi(|x|)-\log (L(|x|)))\right),
$$
and $(L\in RV_\alpha\text{ and }\psi\in RV_v,\ v>0) \Rightarrow \psi(|x|)-\log (L(|x|))\in RV_v, $ hence the tail behavior of $q$ is governed by $\psi$, while the behavior of $L$ becomes asymptotically negligible. Thus without loss of generality, it will suffice to consider light tail densities of the form $q(x)=C\exp(-\psi(|x|))$
, where $C$ is a suitable renormalizing constant. We may relax the conditions imposed on $L$ by seeking the weakest constraints which prevent $L$ from growing exponentially, however we feel that such extra care is superfluous here and does not bring further insight into the overall argument, while adding extra technicalities. Hence we have opted for the simple, yet general requirement that $L\in RV_\alpha$, in our definition of a light tail density. The extra assumption that $\psi'$ is eventually monotone is due to some technicalities required in some of the arguments below, and is not necessary for all cases of sampling densities considered.

\begin{defn}(Decays)\label{defn: decays}
For clarity of the results presented below, it is helpful to organize the different cases of sampling densities defined above into three distinct categories, depending on their decay.
\begin{itemize}
\item We say that a sampling density has \textbf{subexponential decay} if it has a heavy tail or if it has a light tail with $\psi\in RV_v$ with $v\in (0,1)$.
\item We say that the density has \textbf{exponential decay} if it has a light tail with $\psi\in RV_1$.\\
\item We say that the density has \textbf{superexponential decay} if it has a light tail with $\psi\in RV_v$ with $v>1$.
\end{itemize}
\end{defn}


\section{Outline of the results}\label{sec: outline of the results}

In this section we summarize the main results of this work. Going through the different decays for a sampling density defined above, we find conditions on the radius parameter such that random geometric graphs are disconnected w.h.p., or on the other hand in certain cases, such that some concentration inequalities hold.

Recall that given a sampling density $q$, $\mathcal P_n$ is a homogeneous Poisson point process of intensity $nq$ sampled with respect to $q$, and $G(\mathcal P_n,r_n)$ is the random geometric graph, where the radius $r_n\leq 1$ is non-increasing in $n$, and two points of $\mathcal P_n$ are connected by an (undirected) edge if their distance is less than $r_n$. The main results of this work are gathered in the three theorems below.

\begin{thm}[Subexponential decay]\label{thm: subexponential}

     Suppose that $q$ has subexponential decay, then w.h.p. $G(\mathcal P_n,1)$ is disconnected.
     \end{thm}
     
     Since $r_n\leq 1$, if $q$ has subexponential decay, then every random geometric graph is disconnected with high probability in this setting.

\begin{thm}[Exponential decay]\label{thm:exponential}
     Suppose that $q$ has exponential decay and 
    $$
    r_n\psi'(\psi^{\leftarrow}(\log n)) =O(1),
    $$
    then w.h.p. $G(\mathcal P_n,r_n)$ is disconnected.
\end{thm}
    
    If $q$ has exponential decay, then $\psi'\psi^{\leftarrow}(\log n)\sim \psi'(\log n)$ and $\psi'$ is slowly varying. Thus the aymptotic condition imposed on $r_n$ in that case can be rewritten as
$$
r_n\leq C/L(n),
$$
where $C>0$ is a constant and $L$ is slowly varying. This constraint on $r_n$ is essentially trivial in practice, almost equivalent to asking that $r_n$ remains asymptotically constant. Nonetheless, and although this is of limited practical interest, we do not know what happens to the graph if $q$ has exponential decay and
$$
r_n\psi'\psi^{\leftarrow}(\log n)=\omega(1).
$$


In contrast with the two previous cases, if $q$ has superexponential decay, we find a non-trivial asymptotic condition on the radius $r_n$ for the graph to be disconnected. Theorem~\ref{thm:superexponential} below indicates that there exists an asymptotic threshold such that, for values of $r_n$ decaying faster than this threshold, the graph $G(\mathcal P_n,r_n)$ is disconnected w.h.p., while for values of $r_n$ decaying slower than this same threshold, concentration inequalities hold on small cubes partitioning $\R^d$ (which imply in particular the connectivity of the graph). This threshold agrees with the results found in \cite{randGraphs} and \cite{gaussianNetworkConnectivity}, which studied the special case where the density is a Gaussian.

Prior to stating the theorem, we shall need to define a partition of $\R^d$ into small cubes, similarly to \cite{optimalCheeger} in the case of a bounded domain. Let us divide $\R^d$ into a grid of cubes $\{Q'_{i,n}|i\in \N\}$ of side width $\gamma_nr_n$, where $\gamma_n=o(1)$. 
Denote the centre of $Q'_{i,n}$ by $z_{i,n}$ and assume, without loss of generality, that the origin is one of the centres.\\

Given $R_n>0$ define
\begin{equation}\label{index domain partition}
S_n(R_n):=\{i\in \N\text{ }|\text{ }Q'_{i,n}\subset B(0,R_n)\}.
\end{equation}

For $i\in S_n(R_n)$, define $$I(i,n):=\{j\in \N\text{ }|\text{ }\forall i'\in S_n(R_n)\setminus\{i\}\text{, }|z_{i,n}-z_{j,n}|<|z_{i',n}-z_{j,n}|\},$$
and let
\begin{equation}\label{boxes domain partition}
Q_{i,n}:=\cup_{j\in I(i,n)}(Q'_{j,n}\cap B(0,R_n)).
\end{equation}
This forms a partition $\{Q_{i,n}\}_{i\in S_n(R_n)}$ of $B(0,R_n)$ into small cubes.

    \begin{thm}[Superexponential decay]\label{thm:superexponential}
    Suppose that $q$ has superexponential decay.
    \begin{itemize}
    \item  If
    $$
    r_n\psi'(\psi^{\leftarrow}(\log n)) =o(\log \log n),
    $$
    then w.h.p. $G(\mathcal P_n,r_n)$ is disconnected.\\
   \item If 
    $$
    r_n\psi'(\psi^{\leftarrow}(\log n)) =\omega(\log \log n),
    $$
    then there exist increasing sequences of real numbers $(R_n^{(0)})_n$ and $(R_n^{(1)})_n$ such that $R_n^{(1)}-R_n^{(0)}=o(r_n)$,
     and w.h.p.
     
$$     
     \mathcal P_n\cap B(0,R_n^{(1)})^c=\emptyset,
$$   
and for all choices of $\gamma_n=o(1)$ such that $\gamma_n\psi'(\psi^\leftarrow(\log n))=\omega(1)$, for all $i\in S_n(R_n^{(0)})$
\begin{align*}
(1-\gamma_n)n\nu(Q_{i,n})\leq \mathcal P_n(Q_{i,n})\leq (1+\gamma_n)n\nu(Q_{i,n}),
\end{align*}
where $\mathcal P_n(A):=|\mathcal P_n\cap A|$.
\end{itemize}
\end{thm}


Note in particular that the above concentration inequalities imply connectivity of the graph $ G(\mathcal P_n,r_n)$, which could already be deduced by Theorem~$4.5$ of \cite{topoCrackle}. Note also that if the sampling density is a Gaussian, i.e., up to multiplicative constants, say $\psi(z)\sim z^2$, $\psi'(z)\sim z$ and $\psi^\leftarrow(z)\sim z^{1/2}$, then the connectivity threshold obtained from Theorem~\ref{thm:superexponential}, i.e., $\tau\sim\frac{\log\log n}{\sqrt{\log n}}$, agrees with the results of \cite{randGraphs} and \cite{gaussianNetworkConnectivity}.
The authors of \cite{topoCrackle} establish the contractibility of the union of the random balls $\cup_{x\in \mathcal P_n}B(x,r_n)$ under the same asymptotic condition as in Theorem~\ref{thm:superexponential} above:
\begin{equation}\label{connectivity condition}
r_n \psi'\psi^\leftarrow(\log n)=\omega (\log n\log n).
\end{equation}
As explained already in Sections~\ref{sec: intro} and \ref{sec: properties of geometric point processes}, our motivation for presenting a stronger result about concentration properties for the sampled points is that this kind of concentration result is required in order to approximate continuous operators by discrete ones, hence is an important component towards showing the consistency of spectral clustering on $\R^d$. 

We note that the authors in \cite{topoCrackle} asked whether condition~(\ref{connectivity condition}) was sharp in order for the contractibility of $\cup_{x\in \mathcal P_n}B(x,r_n)$ to hold true, or whether it could be weakened to
$$
r_n\psi'\psi^\leftarrow(\log n)=\omega(1).
$$
If this were the case, then in particular the graph $G(\mathcal P_n,r_n)$ would be connected for such values of $r_n$. However Theorem~\ref{thm:superexponential} shows that the graph is disconnected w.h.p. if
$$
r_n\psi'\psi^\leftarrow(\log n)=o(\log \log n),
$$
thus answering to the question asked in \cite{topoCrackle}, showing that condition (\ref{connectivity condition}) is sharp indeed.

We note also that Theorem~\ref{thm:superexponential} does not study the regime $r_n\psi'\psi^\leftarrow(\log n)=\Theta(\log\log n)$. In particular, it would be of interest from the point of view of combinatorics to estimate more precisely the connectivity threshold for random geometric graphs in terms of $\psi$, generalizing the thresholds obtained in the case of a Gaussian sampling density in \cite{randGraphs,gaussianNetworkConnectivity}.
\section{Disconnected Regimes on $\R^d$}\label{disconnected strategy}

The general strategy we follow in order to show that a random geometric graph is disconnected w.h.p., builds on some ideas found in \cite{gaussianNetworkConnectivity}, where the authors investigate connectivity properties of a random geometric graph in the specific case where the vertices are sampled from a Gaussian distribution supported on $\R^2$.

\begin{defn}
We shall use the following notations.
\begin{itemize}
\item Let $R^c$ denote the furthest distance from the origin of a point in $\mathcal P_n$ which is part of the connected component of $ G(\mathcal P_n,r_n)$ containing the origin, and let $R_{\max}$ denote the furthest distance from the origin of a point in $\mathcal P_n$.

\item Given a random geometric graph $G(\mathcal P_n,r_n)$ and an increasing sequence $(R_n)_{n\in \N}$ of real numbers, let $E_{R_n}$ denote the expected number of isolated vertices within distance $R_n$ of the origin.
\end{itemize}
\end{defn}

In order to show that $G(\mathcal P_n,r_n)$ is disconnected w.h.p., it is thus sufficient to show that $R^c<R_{\max}$ w.h.p.. Our strategy will consist in finding two sequences of real numbers $(R_{n}^{(0)})_{n\in \N}$ and $(R_{n}^{(1)})_{n\in \N}$ satisfying\\
\begin{equation}\label{disconnected conditions}
\begin{cases}
R_n^{(0)}\leq R_n^{(1)},\\
\lim_{n\to \infty} E_{R_n^{(0)}}=\infty,\\
\lim_{n\to \infty}\mathbb P\big(\mathcal P_n\cap B(0,R_n^{(1)})^c = \emptyset\big)=0.
\end{cases}
\end{equation}
\begin{lemma}\label{lemma: disconnectivity}
Suppose that there exist sequences $(R_n^{(0)})_{n\in \N}$ and $(R_n^{(1)})_{n\in \N}$ satisfying (\ref{disconnected conditions}), then the graph $G(\mathcal P_n,r_n)$ is disconnected w.h.p..\
\end{lemma}
\begin{proof}
We have the following result, whose proof is a direct generalization of Lemma $6$ in \cite{gaussianNetworkConnectivity}.
\begin{lemma}\label{large component upper bound}
If $\lim_{n\to \infty}E_{R_n}=\infty$, then $R^c\leq R_n$ w.h.p..
\end{lemma}

 By Lemma \ref{large component upper bound}, w.h.p. $R^c\leq R_n^{(0)}$ and we also see that w.h.p. $R_n^{(1)}< R_{\max}$, hence w.h.p.  $R^c<R_{\max}$.
\end{proof}

\section{Heavy Tail Densities}\label{sec: heavy tail}
Suppose throughout this subsection that $r_n\equiv 1$ and that the sampling density has a heavy tail (see Definition \ref{heavy tail}). In this case, we show that $G(\mathcal P_n,1)$ is disconnected w.h.p..  Following the strategy presented in Section~\ref{disconnected strategy}, it suffices to find sequences $(R_n^{(0)})_n$ and $(R_n^{(1)})_n$ satisfying (\ref{disconnected conditions}).

\begin{prop}\label{prop: heavy tail R1}
Let $\epsilon>0$ and suppose that $(R_n)_{n\in \N}$ satisfies
$$
R_n=o(n^{1/(\alpha+\epsilon-d)}),
$$
then
$$
\lim_{n\to \infty}\mathbb P\big(\mathcal P_n\cap B(0,R_n)^c=\emptyset\big)=0.
$$

\end{prop}
\begin{proof}
First note that
$$
    \mathbb P\big(\mathcal P_n\cap B(0,R_n)^c=\emptyset\big)=\exp\left(-n\int_{|x|\geq R_n}q(x)dx\right).
$$
We will show that
$$
\lim_{n\to \infty}n\int_{|x|\geq R_n}q(x)dx = \infty.
$$
Since $q$ is radial, we have
$$
n\int_{|x|\geq R_n}q(x)dx=ns_{d-1}\int_0^\infty(R_n+\rho)^{d-1}q(R_n+\rho)d\rho,
$$
where $s_{d-1}$ denotes the surface area of a unit Euclidean ball in $\R^d$.
Furthermore, $q\in RV_{-\alpha}$ with $\alpha>d$, hence there is a slowly varying function $L$ such that $q(x) = L(|x|)|x|^{-\alpha}$. By virtue of the slow variation of $L$, we deduce that for $|x|$ sufficiently large
$$
q(x)>|x|^{-\alpha-\epsilon},
$$
hence
\begin{align*}
    n\int_{|x|\geq R_n}q(x)dx&>ns_{d-1}\int_0^\infty(R_n+\rho)^{d-1-\alpha-\epsilon}d\rho\\
    &=\frac{ns_{d-1}}{d-1-\alpha-\epsilon}\left[\big(R_n+\rho\big)^{d-\alpha-\epsilon}\right]_0^\infty\\
    &=\frac{s_{d-1}}{\alpha+\epsilon+1-d}\big(n{R_n}^{d-\alpha-\epsilon}\big),
\end{align*}
which is $\omega(1)$, using the assumptions on $R_n$ and that $\alpha> d$.
\end{proof}
\begin{prop}\label{prop: heavy tail R0}
Suppose that
$$
R_n=\omega( n^{1/(\alpha-\epsilon)}),
$$
where $0<\epsilon<d$,
then
$$
\lim_{n\to \infty}E_{R_n} = \infty.
$$
\end{prop}
\begin{proof}
Given $y$ such that $|y|=\rho$,
\begin{align*}
\mathbb P\big(y \not\in \cup_{x\in \mathcal P_n}B(x,1)\big)&=\mathbb P\big(\mathcal P_n \cap B(y,1) = \emptyset\big)\\
&=\exp\left(-n\int_{B(y,1)}q(x)dx\right)\\
&\geq\exp\left(-n\omega_dq(\rho+\delta)\right),
\end{align*}
where $\delta\in (-1,1)$ is such that $q(\rho+\delta):=\max\{q(z)\ |\ z\in B(y,1)\}$.
Hence
\begin{align*}
    E_{R_n}&=\int_{B(0,R_n)}\mathbb P\big(y\not\in \cup_{x\in \mathcal P_n}B(x,1)\big)d\mathbb P(y)\\
    &\geq ns_{d-1}\int_{R_n/2}^{R_n}\exp\left(-n\omega_d q(\rho+\delta)\right)\rho^{d-1}q(\rho)d\rho.
\end{align*}
Since $q\in RV_{-\alpha}$, we have for $\rho$ sufficiently large
$$
q(\rho+\delta)\leq (\rho+\delta)^{-\alpha+\epsilon}
$$
and
$$
\rho^{d-1}q(\rho)\geq \rho^{d-1-\alpha-\epsilon'},
$$
where $\epsilon'>0$ is such that $\epsilon+\epsilon'<d$. In particular, note from our choice of $\epsilon'$ that 
\begin{equation}\label{eq: upper bound condition on R_n}
    R_n=o(n^{1/(\alpha+\epsilon'-d)}).
\end{equation}
Thus
$$
E_{R_n}\gtrsim \exp\left(-n\omega_d(R_n/2+\delta)^{-\alpha+\epsilon}\right)n\int_{R_n/2}^{R_n}\rho^{d-1-\alpha-\epsilon'}d\rho.
$$
We have
$$
n(R_n/2+\delta)^{-(\alpha-\epsilon)}=o(1),
$$
hence
$$
\exp\left(-n\omega_d(R_n/2+\delta)^{-\alpha+\epsilon}\right)\sim 1,
$$
hence, using (\ref{eq: upper bound condition on R_n})
\begin{align*}
    E_{R_n}&\gtrsim \frac{n}{d-1-\alpha-\epsilon'}\left[\rho^{d-\alpha-\epsilon'}\right]_{R_n/2}^{R_n}\\
    &\gtrsim nR_n^{d-\alpha-\epsilon'}\\
    &=\omega(1).
\end{align*}

\end{proof}

\begin{corollary}\label{cor: heavy tail disconnectivity}
If $q$ has a heavy tail then $ G(\mathcal P_n,1)$ is disconnected w.h.p..
\end{corollary}
\begin{proof}
Let $\epsilon:=d/2$, and let $R_n^{(0)}=R_n^{(1)}\sim n^{1/(\alpha-\epsilon)}$. By Proposition~\ref{prop: heavy tail R0}, we know that $\lim_{n\to\infty} E_{R_n^{(0)}}=\infty$.

Letting $\epsilon':=d/4$, such that $\epsilon+\epsilon'<d$, we see that $R_n^{(1)}=o(n^{1/(\alpha+\epsilon'-d)})$. By Proposition~\ref{prop: heavy tail R1} then, 
$$\lim_{n\to \infty}\mathbb P\big(\mathcal P_n\cap B(0,R_n)^c=\emptyset\big)=0.$$ We can thus conclude, by Lemma~\ref{lemma: disconnectivity}, that the graph $ G(\mathcal P_n,1)$ is disconnected w.h.p..
\end{proof}

\section{Light Tail Densities}
Suppose that $q$ is of the form
\begin{equation}\label{light tail form}
q(x) = L(|x|)\exp\left(-\psi(|x|)\right),
\end{equation}
where asymptotically $L\equiv C$, $\psi$ is of von Mises type, $\psi'$ is eventually monotone and $\psi\in RV_v$ with $v>0$.\\
Define $(R_n^{(0)})_n$ and $(R_n^{(1)})_n$ by
$$
\psi(R_n^{(0)}):=\log n
$$
and
$$
\psi\big(R_n^{(1)}\big):=\log n+(d-1)\log \psi^{\leftarrow}(\log n)-\log \big(\psi'( \psi^{\leftarrow}(\log n))\big)-w(n),
$$
where $w(n)\to \infty$ as $n\to \infty$, and such that $w(n)=o(\log\log n)$, and where we recall that $\psi^\leftarrow$ denotes the asymptotic inverse function of $\psi$ (which is well-defined since $\psi$ is of von Mises type). 

Going through the various decays for a light tail density, we find for each case asymptotic conditions on the radius parameter $r_n$ such that the sequences $(R_n^{(0)})_n$ and $(R_n^{(1)})_n$ satisfy (\ref{disconnected conditions}), hence such that the induced graph $G(\mathcal P_n,r_n)$ is disconnected w.h.p.. We already have the following result for $R_n^{(1)}$.
\begin{prop}\label{R_n^(1) light tail asymptotic}
We have 
$$
\mathbb P\left(\mathcal P_n\cap B(0,R_n^{(1)})^c=\emptyset\right)= e^{-\alpha(n)},
$$
where $\alpha(n)\sim e^{w(n)}$.
\end{prop}
\begin{proof}
Denote $R:=R_n^{(1)}$ to simplify the notation in the proof. We have
$$
\mathbb P\left(\mathcal P_n\cap B(0,R_n^{(1)})^c\right)=\exp\left(-n\int_{|x|\geq R}q(x) dx\right),
$$
and recalling that $q$ is radial, appealing to a spherical change of coordinates, then to the change of variables $z \leftrightarrow R+z/\psi'(R)$, we have
\begin{align}\label{integral}
 n\int_{|x|\geq R}q(x)dx=&s_{d-1}n\int_{R}^\infty z^{d-1}q(z)dz\\
    &=s_{d-1}n\frac{R^{d-1}}{\psi'(R)}q(R)\\ \nonumber
    &\times \int_0^\infty\left(1+\frac{z}{\psi'(R)R}\right)^{d-1}\frac{q(R+z/\psi'(R))}{q(R)}dz,
\end{align}
where $s_{d-1}$ denotes the surface area of a Euclidean unit sphere in $\R^d$.\\
Let us estimate the integral on the RHS of (\ref{integral}), as $n\to \infty$. By the mean value theorem, there exists $t$ between $R$ and $R+z/\psi'(R)$ such that
$$
\psi(R)-\psi(R+z/\psi'(R))=-\frac{\psi'(t)}{\psi'(R)}z.
$$
Since $\psi'(R)^{-1}=o(R)$, then $t\sim R$ and by the regular variation of $\psi'$
$$
\psi'(t)\sim \psi'(R);
$$
hence as $n\to \infty$,
$$
\psi(R)-\psi(R+z/\psi'(R))\to -z,
$$
from which we find by the dominated convergence theorem that the integral on the RHS above converges as $n\to \infty$ to
$$
\int_0^\infty e^{-z}dz=1.
$$
For the remaining factor on the RHS of (\ref{integral}), we have by the assumptions on $R$
$$
n\frac{R^{d-1}}{\psi'(R)}e^{-\psi(R)}\sim e^{w(n)}.
$$

\end{proof}

\begin{corollary}\label{cor:light tail R_n^(1) choice}
We have $R_n^{(0)}\leq R_n^{(1)}$ and
$$
\lim_{n\to \infty}\mathbb P\left(\mathcal P_n\cap B(0,R_n^{(1)})^c=\emptyset\right)=0.
$$
\end{corollary}
In light of Lemma~\ref{lemma: disconnectivity}, it remains only to find conditions on the radius parameter of the graph $r_n$ such that $\lim_{n\to \infty}E_{R_n^{(0)}}=\infty$, in order to conclude that $G(\mathcal P_n,r_n)$ is disconnected w.h.p.. Recall that, given a graph $ G(\mathcal P_n,r_n)$, $E_{R_n}$ denotes the expected number of isolated vertices contained in $B(0,R_n).$

\subsection{Subexponential decay}
\begin{prop}\label{R_n^(0) subexp}
Suppose that $q$ is of the form (\ref{light tail form}) with $\psi\in RV_v$ and $v\in (0,1)$, and suppose that $r_n\equiv 1$, then
$$
\lim_{n\to \infty}E_{R_n^{(0)}} = \infty.
$$
\end{prop}
\begin{proof}
To simplify notation in the proof, let $R:=R_n^{(0)}$. We have
\begin{align*}
    E_R&\gtrsim n\int_0^R\rho^{d-1}e^{-\psi(\rho)}e^{-n\omega_d e^{-\psi(\rho-1)}}d\rho\\
    &\gtrsim n\int_{\alpha(n)}^{R-1} \rho^{d-1}e^{-\psi(\rho+1)}e^{-n\omega_d e^{-\psi(\rho)}}d\rho,
\end{align*}
where $\alpha(n)\to \infty$ arbitrarily slowly as $n\to \infty$.\\

By the mean value theorem, for all $\rho\in [\alpha(n),R-1]$, there exists $t\in [\rho,\rho+1]$ such that
$$
\psi(\rho+1)-\psi(\rho) = \psi'(t),
$$
and $\psi'\in RV_{v-1}$ with $v-1<0$, hence $\lim_{n\to \infty}\psi'(t) = 0$. In particular, for $n$ sufficiently large, we have for all $\rho \in [\alpha(n),R-1]$
$$
\psi(\rho+1)\leq \psi(\rho) + 1,
$$
hence
$$
E_R\gtrsim n\int_{\alpha(n)}^{R-1}\rho^{d-1}e^{-\psi(\rho)}e^{-n\omega_d e^{-\psi(\rho)}}d\rho.
$$
We can write $\psi'\in RV_{v-1}$ in the form
$$
\psi'(\rho) = L(\rho)\rho^{v-1},
$$
where $L$ is slowly varying. We know by assumptions that $v< 1<d$, hence we can write
$$
\rho^{d-1} = (L(\rho))^{-1}\rho^{d-v}\psi'(\rho),
$$
and we find that
\begin{align*}
E_R&\gtrsim (L(\alpha(n)))^{-1}\alpha(n)^{d-v}\int_{\alpha(n)}^{R-1}\big(n\psi'(\rho)e^{-\psi(\rho)}\big)e^{-n\omega_d e^{-\psi(\rho)}}d\rho\\
&\gtrsim (L(\alpha(n)))^{-1}\alpha(n)^{d-v}\left[\exp\big(-n\omega_de^{-\psi(\rho)}\big)\right]_{\alpha(n)}^{R-1}\\
&\gtrsim (L(\alpha(n)))^{-1}\alpha(n)^{d-v}\left(\exp\big(-n\omega_de^{-\psi(R-1)}\big)-\exp\big(-n\omega_de^{-\psi(\alpha(n))}\big)\right).
\end{align*}
Since $(L(\alpha(n)))^{-1}\alpha(n)^{d-v}=\omega(1)$, in order to show that $E_R=\omega(1)$, it suffices to show that the sum of terms in the parentheses above is $\Theta(1)$. With $\alpha(n)$ growing sufficiently slowly as $n\to \infty$, we already have that
$$
\exp\big(-n\omega_de^{-\psi(\alpha(n))}\big)=o(1).
$$
As before by the mean value theorem, there exists $t_0\in [R-1,R]$ such that
$$
\psi(R)=\psi(R-1)+\psi'(t_0),
$$
but $\psi'(t_0)=o(1)$ (since $\psi'\in RV_{v-1}$ and $v<1$ by assumptions) and by construction of $R=R_n^{(0)}$, $e^{-\psi(R)}=n^{-1}$, hence
$$
n\omega_de^{-\psi(R-1)}=\Theta(e^{-\psi'(t_0)}) = \Theta(1).
$$
\end{proof}

This result, together with Section~\ref{sec: heavy tail}, completes the proof of Theorem~\ref{thm: subexponential}. If $q$ has subexponential decay (i.e., a heavy-tail or a light tail with $\psi\in RV_v$ and $v\in (0,1)$), then the graph $G(\mathcal P_n,1)$ is disconnected w.h.p..

\subsection{Exponential decay}
\begin{prop}\label{exp R_n^(0)}
Suppose that $q$ is of the form (\ref{light tail form}) with $\psi\in RV_1$. 
If
$$
r_n \psi'(R_n^{(0)}) = O(1),
$$
then 
$$
\lim_{n\to \infty}E_{R_n^{(0)}} = \infty.
$$
\end{prop}
\begin{proof}
Let $R:= R_n^{(0)}$ and $r:=r_n$. 
The beginning of the proof proceeds as before. We have
$$
E_R\gtrsim n\int_{\alpha(n)}^{R-r} \rho^{d-1}e^{-\psi(\rho+r)}e^{-n\omega_dr^d e^{-\psi(\rho)}}d\rho,
$$
by the mean value theorem, there exists $t\in [\rho,\rho+r]$ such that
$$
\psi(\rho+r) = \psi(\rho)+r\psi'(t)\leq \psi(\rho)+r\psi'(R),
$$
and
$$
\rho^{d-1} \geq\psi'(\rho),
$$
hence
\begin{align*}
E_R&\gtrsim r^{-d}e^{-r\psi'(R)}\int_{\alpha(n)}^{R-r} \big(nr^d\psi'(\rho)e^{-\psi(\rho)}\big)e^{-n\omega_d r^d e^{-\psi(\rho)}}d\rho\\
&\gtrsim r^{-d}e^{-r\psi'(R)}\left(\exp\big(-n\omega_dr^de^{-\psi(R-r)}\big)-\exp\big(-n\omega_dr^de^{-\psi(\alpha(n))}\big)\right).
\end{align*}
Using the assumptions on $r$, we find that
$$
\lim_{n\to \infty}r^{-d}e^{-r\psi'(R)} = \infty,
$$
that
$$
n\omega_dr^de^{-\psi(R-r)} = o(1),
$$
and for $\alpha(n)$ growing sufficiently slowly
$$
 n\omega_dr^de^{-\psi(\alpha(n))}=\omega(1).
$$
Thus $E_R=\omega(1)$.

\end{proof}

This completes the proof of Theorem~\ref{thm:exponential}. If $q$ has exponential decay and
$$
r_n\psi'(\log n) = O(1),
$$
then the graph $G(\mathcal P_n,r_n)$ is disconnected w.h.p..



\subsection{Superexponential decay}
\begin{prop}\label{expected number of isolated vertices}
Suppose that $q$ satisfies (\ref{light tail form}) and that $\psi\in RV_v$ with $v>1$. If
$$
r_n\psi'(R_n^{(0)}) = o(\log \log n),
$$
then $$\lim_{n\to \infty} E_{R_n^{(0)}} = \infty.$$
\end{prop}
\begin{proof}
Let $\varphi(z)=\min\{z,0\}$, $z\in \R$, and let $R:=R_n^{(0)}$ and $r:=r_n$.

We have
\begin{align*}
    E_R&\gtrsim n\int_{r}^R\delta^{d-1}e^{-\psi(\delta)}e^{-n\omega_dr^de^{-\psi(\delta -r)}}d\delta\\
    &\gtrsim n \int_{\alpha(n)}^{R-r}(\delta+r)^{d-1}e^{-\psi(\delta+r)}e^{-n\omega_d r^de^{-\psi(\delta)}}d\delta, \\
    \end{align*}
    where $\alpha(n)\to \infty$ arbitrarily slowly as $n\to \infty$.
    For every $\delta\in [\alpha(n),R-r]$
    $$
    \delta^{d-1}\geq\delta^{v-1}R^{\varphi(d-v)},
    $$
    furthermore, we know by Proposition $2.5$ in \cite{heavyTail} that the derivative of $\psi\in RV_v$ satisfies
    $$
    \psi'\in RV_{v-1},
    $$
    hence there exists a slowly varying function $L$ such that for all $\delta\in \R_+$
    $$
    \psi'(\delta):=\delta^{v-1}L(\delta)\leq\delta^{v-1+\epsilon} ,
    $$
    for all $\epsilon>0$.
    We then find for all $\epsilon>0$
    $$
    \delta^{d-1}\geq R^{\varphi(d-v)-\epsilon}\psi'(\delta),
    $$
    and so
    \begin{align}
    E_R&\gtrsim r^{-d}e^{-r\psi'(R)(1+o(1))}R^{\varphi(d-v)-\epsilon}\int_{\alpha(n)}^{R-r}\left(nr^d\psi'(\delta)e^{-\psi(\delta)}\right)e^{-n\omega_d r^de^{-\psi(\delta)}}d\delta\nonumber\\
    &\gtrsim r^{-d}e^{-r\psi'(R)(1+o(1))}R^{\varphi(d-v)-\epsilon}\left(e^{-n\omega_d r^de^{-\psi(R-r)}}-e^{-n\omega_d r^de^{-\psi(\alpha(n))}}\right). \label{eqn}
    \end{align}
    We now show that the sum of terms in the parentheses in (\ref{eqn}) tends to $1$ as $n\to \infty$, while the remaining factor, on the left of the parentheses in (\ref{eqn}), tends to $\infty$.\\
    
    Using
    \begin{align*}
    \psi(R-r) &= \psi(R)-r\psi'(R)(1+o(1))\\
    &=\log n - r\psi'(R)(1+o(1)),
    \end{align*}
    we have
    $$
        n\omega_d r^de^{-\psi(R-r)} \sim \exp\big( d\log r + r\psi'(R)(1+o(1))\big);
    $$
    by assumptions
    $$
    r\psi'(R) = o(\log \log n),
    $$
    while
    $$
    d\log (r) \lesssim -\log \psi'(R) \lesssim -\log \log n,
    $$
    and so
    $$
    \lim_{n\to \infty}n\omega_d r^de^{-\psi(R-r)}=0
    $$
    and
    $$
    \lim_{n\to \infty}e^{-n\omega_d r^de^{-\psi(R-r)}}=1,
    $$
    and also, for $\alpha(n)$ growing sufficiently slowly,
    $$
    \lim_{n\to \infty}e^{-n\omega_d r^de^{-\psi(\alpha(n))}}=0.
    $$
    Finally, assuming without loss of generality that $\varphi(d-v) = d-v\leq 0$, we have
    \begin{align*}
        r^{-d}e^{-r\psi'(R)(1+o(1))}R^{\varphi(d-v)-\epsilon}&\gtrsim R^{d(v-1)+d-v-2\epsilon}e^{-\psi'(R)(1+o(1))}\\
        &\gtrsim \exp\big((v(d-1)-2\epsilon)\log \log n-r\psi'(R)(1+o(1))\big)\\
        &=\omega(1).
    \end{align*}
   
    Wrapping up the above, we deduce that
    $$
    \lim_{n\to \infty} E_R = \infty.
    $$

\end{proof}
This completes the proof of the first part of Theorem~\ref{thm:superexponential}. If $q$ has superexponential decay and
$$
r_n\psi'(\psi^\leftarrow(\log n)) = o(\log \log n),
$$
then the graph $G(\mathcal P_n,r_n)$ is disconnected w.h.p..


\section{Concentration Regime}
The second part of Theorem~\ref{thm:superexponential} is concerned with providing asymptotic conditions such that the sampled points satisfy some concentration properties.
These kinds of concentration inequalities are well-known in the case of bounded domains (e.g., Theorem $1.1$ in \cite{WassersteinRate} or Lemma $3.2$ in \cite{optimalCheeger}), but do not have an immediate analogue on $\R^d$. This is due to the fact that these results rely on the construction of a suitable partition of the domain into small cubes centered around the sampled points; however if the domain is unbounded, no such partition can be found for any given $n$, where there are only finitely many points in $\mathcal P_n$. For instance, the $\infty$-Wasserstein distance satisfies $\infty=d_H(\mathcal P_n,\R^d)\leq W_{\infty}(\nu_n,\nu),$ where $d_H$ denotes the Hausdorff distance, and $\nu_n$ is the empirical measure with respect to $\mathcal P_n$.

Instead, we shall restrict our attention to $B(0,R_n^{(0)})$ for a suitable increasing sequence of positive real numbers $(R_n^{(0)})_{n\in \N}$, such that for every $n\in \N$, concentration inequalities (analogous to Theorem $1$ in \cite{WassersteinRate} or Lemma $3.2$ in \cite{optimalCheeger}) hold for the sampled points contained in $B(0,R_n^{(0)})$, while very few sampled points exist outside of the ball.  

Recall the partition of $B(0,R_n^{(0)})$ into small cubes $\{Q_{i,n}\}_{i\in S(R_n^{(0)})}$ of side width $\gamma_n r_n$, constructed in Section~\ref{sec: outline of the results} (see definition in (\ref{boxes domain partition})).
We have the following general result, which can be viewed as a stronger version of Theorem $4.5$ in \cite{topoCrackle}.

\begin{thm}\label{unbounded discrepancy}
Suppose that the sampling density $q$ has superexponential decay (see Definition~\ref{defn: decays}). Suppose furthermore that $(r_n)_n$ satisfies
$$
r_n\psi'(\psi^{\leftarrow}(\log n)) = \omega(\log \log n),
$$
and that $\gamma_n=o(1)$ is such that
$$
\gamma_nr_n\psi'(\psi^\leftarrow(\log n))=\omega(1).
$$
Then, there exists sequences $(R_n^{(0)})_{n\in \N}$ and $(R_n^{(1)})_{n\in \N}$ such that
$$
R_n^{(1)}-R_n^{(0)}=o(r_n),
$$
and such that with probability going to $1$ as $n\to \infty$,
$$
\mathcal P_n\cap B(0,R_n^{(1)})^c = \emptyset,
$$
and for all $i\in S_n(R_n^{(0)})$
\begin{align*}
(1-\gamma_n)n\nu(Q_{i,n})\leq \mathcal P_n(Q_{i,n})\leq (1+\gamma_n)n\nu(Q_{i,n}).
\end{align*}
\end{thm}
Before we give a proof of the theorem, we shall need to recall some Chernoff-type bounds for Poisson random variables (see Lemma $1.2$ in \cite{randGraphs}).

Let $N\sim Po(n)$ be a Poisson random variable with intensity $n$. For $x>0$ let
$$
H(x):=1-x+x\log x,
$$
and set $H(0):=1$.
We have the following Chernoff-type bounds
\begin{align*}
    \mathbb P\left(N\geq k\right)&\leq \exp\left(-nH\left(\frac{k}{n}\right)\right)\text{, }k\geq n;\\
    \mathbb P\left(N\leq k\right)&\leq \exp\left(-nH\left(\frac{k}{n}\right)\right)\text{, }k\leq n.
\end{align*}
\begin{proof}
As discussed in Section~\ref{sec: probability measures}, we may assume that asymptotically, $q$ is of the form $q(x)=C \exp(-\psi(|x|))$, where $C>0$ is a suitably normalizing constant, and $\psi\in RV_v$, $v>1$.

Let $R_n^{(0)}:=\psi^\leftarrow (A_n)$, where
$$
A_n:=\log n+d\log(r_n)+(d+2)\log(\gamma_n)-\log\log((\gamma_nr_n)^{-1}\psi^{\leftarrow}(\log n))-\delta,
$$
and $\delta$ is chosen such that
$$
(d-Ce^{\delta}/3)<0,
$$
and let $R_n^{(1)}:=\psi^\leftarrow(B_n)$, where
$$
B_n:=\log n+(d-1)\log \psi^\leftarrow(\log n)-\log\psi'(\psi^\leftarrow(\log n))+\log\log n.
$$
\begin{itemize}
\item One can verify as in the proof of Theorem $4.5$ in \cite{topoCrackle}, that w.h.p.
$$
\mathcal P_n\cap B(0,R_n^{(1)})^c=\emptyset,
$$
the choice of $R_n^{(1)}$ being the same.

\item Next, we show that $R_n^{(1)}-R_n^{(0)}=o(r_n)$. 

We have $B_n=\log n +O(\log \log n)$. This is because $\psi^\leftarrow \in RV_{1/v}$ and $\psi'\circ\psi^\leftarrow \in RV_{(1-v)/v}$, hence $\log \psi^\leftarrow (\log n) $ and $\log \psi'(\psi^\leftarrow(\log n))$ are $O(\log \log n).$

Likewise $A_n=\log n +O(\log \log n)$, since
\begin{align*}
(1<\psi'(\psi^\leftarrow(\log n)) \gamma_n r_n)&\Rightarrow (\log (1/\psi'(\psi^\leftarrow(\log n)))<\log \gamma_nr_n<0)\\
&\Rightarrow \log(r_n^d\gamma_n^{d+2})=O(\log \log n).
\end{align*}
Hence $B_n-A_n=O(\log\log n)$, and by the mean value theorem
$$
    R_n^{(1)}-R_n^{(0)}\sim(\psi^\leftarrow)'(\log n)(B_n-A_n)=o(r_n),
$$
since, by our assumptions,
$$
(\psi^\leftarrow)'(\log n)=\frac{1}{\psi'(\psi^\leftarrow(\log n))}=o(\frac{r_n} {\log\log n}).
$$

\item Finally, we show the concentration inequalities for the sampled points. 

Let $i\in S_n(R_n^{(0)}).$ By construction of $Q_{i,n}$, there exists a cube $Q'_{i,n}$ of side $\gamma_nr_n$, entirely contained in $B(0,R_n^{(0)})$ and such that
$
Q'_{i,n}\subset Q_{i,n}.
$
Hence
$$
\nu(Q_{i,n})\geq \nu(Q'_{i,n})\geq q(R_n^{(0)})(\gamma_nr_n)^d.
$$
We may Taylor expand $H$ around $1$ and find for sufficiently small $\abs{x}$, that
$$
H(1+x) > x^2/3.
$$
Thus, picking $n$ sufficiently large so that $H(1+\gamma_n)>(1/3)\gamma_n^2$, and using the above Chernoff-type bounds, we have for all $i\in S_n(R_n)$
\begin{align*}
    \mathbb P\left(\mathcal P_n(Q_{i,n})\geq(1+\gamma_n)n\nu(Q_{i,n})\right)&\leq \exp\left(-n\nu(Q_{i,n})H(1+\gamma_n)\right)\\
    &\leq\exp\left(-nq(R_n)r_n^d\gamma_n^{d+2}/3\right).
\end{align*}
The same bound can be found for $$
\mathbb P\left(\mathcal P_n(Q_{i,n})\leq(1-\gamma_n)n\nu(Q_{i,n})\right),
$$
thus, we find by a union bound
\begin{align*}
    \mathbb P\big(\exists i\in S_n(R_n^{(0)}),\ & \abs{\mathcal P_n(Q_{i,n})-n\nu(Q_{i,n})}>n\gamma_n \nu(Q_{i n})) \big)\\
    &\lesssim \left(\frac{R_n^{(0)}}{\gamma_nr_n}\right)^d\exp\left(-nq(R_n^{(0)})r_n^d\gamma_n^{d+2}/3\right).
\end{align*}
We have 
$$
\exp\left(-nq(R_n^{(0)})r_n^d\gamma_n^{d+2}/3\right) = \frac{Ce^\delta}{3}\log((\gamma_nr_n)^{-1}\psi^\leftarrow(\log n)),
$$
hence 
\begin{align*}
\left(\frac{R_n^{(0)}}{\gamma_nr_n}\right)^d\exp\left(-nq(R_n^{(0)}e_1)r_n^d\gamma_n^{d+2}/3\right)&=\exp\left((d-Ce^\delta/3)\log((\gamma_nr_n)^{-1}\psi^\leftarrow(\log n))\right)\\
&=o(1).
\end{align*}
\end{itemize}
\end{proof}

\section{Conclusion}
In this work we analysed connectivity properties of random geometric graphs on $\R^d$, spanning through the various possible decays of the sampling density (subexponential, exponential and superexponential). We found a dichotomy between the setting where the decay is exponential or slower, and the setting where the decay is faster than exponential. In the former case, we showed that random geometric graphs are disconnected w.h.p.\ under no or very mild constraints on the radius parameter of the graph. In the later case, we found a non-trivial threshold $\tau=\frac{\log\log n}{\psi'\psi^\leftarrow(\log n)}$, such that some concentration properties hold if $r_n=\omega(\tau)$, while the graph is disconnected w.h.p. if $r_n=o(\tau)$. We note that we have not given the exact threshold connectivity value for random geometric graphs, which occurs in the regime $r_n=\Theta(\tau)$. This regime remains to be studied.

Our analysis was motivated by the need to uncover a suitable setting under which geometric learning problems on $\R^d$, such as spectral clustering, can be shown to achieve consistency. The fact that random geometric graphs are typically disconnected if the sampling density has exponential decay or slower, suggests that we should focus our attention on the special setting where sampling densities have superexponential decay. Indeed, it is customary in many geometric learning problems to require the underlying graph to be connected. In this case it would be fruitless to study spectral clustering algorithms set on $\R^d$ unless the sampling density is assumed to have superexponential decay (e.g., a Gaussian). To our knowledge, it was not known before that a large of class of sampling densities on $\R^d$ were not advisable for spectral clustering. An important next step is to focus on the setting where the sampling density has superexponential decay, and find conditions on the radius parameter $r_n$ such that spectral clustering algorithms are consistent on $\R^d$.

\newpage
\bibliographystyle{plain}
\bibliography{refs}

\begin{thebibliography}{10}

\bibitem{crackle}
Robert~J. Adler, Omer Bobrowski, and Samuel Weinberger.
\newblock Crackle: The homology of noise.
\newblock {\em Discrete and Computational Geometry}, 52:680 -- 704, 2014.

\bibitem{randomGraphsSurvey}
Paul Balister, B\'{e}la Bollob\'{a}s, and Amites Sarkar.
\newblock Percolation, connectivity, coverage and colouring of random geometric
  graphs.
\newblock In Paul Balister, Amites Sarkar, and B\'{e}la Bollob\'{a}s, editors,
  {\em Handbook of Large-Scale Random Networks}, volume~18, pages 117 -- 142.
  Springer-Verlag Berlin Heidelberg, 2008.

\bibitem{gaussianNetworkConnectivity}
Paul Balister, Bel\'{a} Bollob\'{a}s, Amites Sarkar, and Mark Walters.
\newblock Connectivity of a gaussian network.
\newblock {\em International Journal of Ad Hoc and Ubiquitous Computing},
  3(3):204 -- 213, 2008.

\bibitem{highRiskExtreme}
G.~Balkema and P.~Embretchs.
\newblock High risks scenarios and extremes: A geometric approach.
\newblock {\em European Mathematical Society}, 2007.

\bibitem{regularVariation}
N.~H. Bingham, C.~M. Goldie, and I.~L. Teugels.
\newblock {\em Regular Variation}.
\newblock Cambridge, 1989.

\bibitem{randomPlaneNetworks}
{E. N.} Gilbert.
\newblock Random plane networks.
\newblock {\em Journal of the Society for Industrial and Applied Mathematics},
  9(4):533 -- 543, 1961.

\bibitem{sphereRandomCovering}
{E. N.} Gilbert.
\newblock The probability of covering a sphere with $n$ circular caps.
\newblock {\em Biometrika}, 52(3/4):323 -- 330, 1965.

\bibitem{spectralClusteringTutorial}
{Ulrike von} Luxburg.
\newblock A tutorial on spectral clustering.
\newblock {\em Statistics and Computing}, 17(4):395 -- 416, 2007.

\bibitem{consistencySpectralClustering}
{Ulrike von} Luxburg, Mikhail Belkin, and Olivier Bousquet.
\newblock Consitency of spectral clustering.
\newblock {\em The Annals of Statistics}, 36(2):555 -- 586, 2008.

\bibitem{optimalCheeger}
Tobias Müller and Mathew Penrose.
\newblock Optimal {C}heeger cuts and bisections of random graphs.
\newblock {\em arXiv}, 2018.

\bibitem{topoCrackle}
Takashi Owada and Robert~J. Adler.
\newblock Limit theorems for point processes under geometric constraints (and
  topological crackle).
\newblock {\em Ann. Probab.}, 45(3):2004--2055, 2017.

\bibitem{randomGraphsConnectivity}
Mathew Penrose.
\newblock The longest edge of the random minimal spanning tree.
\newblock {\em Annals of Applied Probability}, 7(2):340 -- 361, 1997.

\bibitem{randomGraphsConnectivityStrongLaw}
Mathew Penrose.
\newblock A strong law for the longest edge of the minimal spanning tree.
\newblock {\em Annals of Probability}, 27(1):246 -- 260, 1999.

\bibitem{randGraphs}
Mathew Penrose.
\newblock {\em Random Geometric Graphs}, volume~5 of {\em Oxford Studies in
  Probability}.
\newblock Oxford University Press, Oxford, 2003.

\bibitem{heavyTail}
S.~I. Resnik.
\newblock {\em Heavy-Tail Phenomena: Probabilistic and Statistical Modeling}.
\newblock Springer, New York, 2007.

\bibitem{regularlyVaryingSequences}
J.~Salambos and E.~Seneta.
\newblock Regularly varying sequences.
\newblock {\em Proceedings of the American Mathematical Society}.

\bibitem{errorEstimates}
N.~G. Trillos, M.~Gerlach, M.~Hein, and D.~Slep\v{c}ev.
\newblock Error estimates for spectral convergence of the graph {L}aplacian on
  random geometric graphs towards the {L}aplace--{B}eltrami operator.
\newblock {\em arXiv}, 2018.

\bibitem{WassersteinRate}
N.~G. Trillos and D.~Slep\v{c}ev.
\newblock On the rate of convergence of empirical measures in
  $\infty$-transportation distance.
\newblock {\em Canadian Journal of Mathematics}, 67(6):1358 -- 1383, 2015.

\bibitem{variationalApproach}
N.~G. Trillos and D.~Slep\v{c}ev.
\newblock A variational approach to the consistency of spectral clustering.
\newblock {\em Applied and Computational Harmonic Analysis}, 45(2):239 -- 281,
  2018.

\end{thebibliography}
\newpage

\end{document}